\documentclass[11pt]{amsart}

\makeatletter
\@namedef{subjclassname@2020}{\emph{2020} Mathematics Subject Classification}
\makeatother

\usepackage[margin=1in]{geometry} 
\usepackage{amsmath, amsthm, amsfonts, amssymb} 
\usepackage{mathrsfs} 
\usepackage{enumerate} 
\usepackage{xcolor} 
\usepackage{bbm} 
\usepackage{hyperref} 
\usepackage{bm} 
\usepackage{tikz} 
\usetikzlibrary{cd} 
\tikzcdset{scale cd/.style={every label/.append style={scale=#1},
    cells={nodes={scale=#1}}}}

\usepackage{mdframed}

\usepackage{dsfont}

\usepackage{chngcntr}
\counterwithin{equation}{section} 

\usepackage{thmtools, thm-restate} 

\theoremstyle{theorem}
	\newtheorem{theorem}{Theorem}[section]
	
	\newtheorem{proposition}[theorem]{Proposition}
	
	\newtheorem{premise}[theorem]{Premise}
\theoremstyle{definition}
    \newtheorem{definition}[theorem]{Definition}
	\newtheorem{remark}[theorem]{Remark}
    \newtheorem{example}[theorem]{Example}

\newcommand{\N}{\mathbb{N}}
\newcommand{\Z}{\mathbb{Z}}
\newcommand{\Q}{\mathbb{Q}}
\newcommand{\R}{\mathbb{R}}
\newcommand{\C}{\mathbb{C}}
\newcommand{\T}{\mathbb{T}}

\newcommand{\F}{\mathbb{F}}

\newcommand{\mZ}{\mathcal{Z}}
\newcommand{\mQ}{\mathcal{Q}}
\newcommand{\mR}{\mathcal{R}}
\newcommand{\mT}{\mathcal{T}}

\newcommand{\mS}{\mathcal{S}}

\newcommand{\mF}{\mathcal{F}}

\newcommand{\mO}{\mathcal{O}}

\newcommand{\mH}{\mathcal{H}}

\newcommand{\eps}{\varepsilon}

\newcommand{\es}{\emptyset}

\renewcommand{\tilde}{\widetilde}

\title{A note on polynomial equidistribution and recurrence in finite characteristic}

\author{Ethan Ackelsberg}
\address{Institute of Mathematics, \'{E}cole Polytechnique F\'{e}d\'{e}rale de Lausanne (EPFL), Lausanne, CH-1015}
\email{ethan.ackelsberg@epfl.ch}

\author{Vitaly Bergelson}
\address{Department of Mathematics, Ohio State University, Columbus, OH 43210}
\email{vitaly@math.ohio-state.edu}

\date{\today}



\begin{document}

\maketitle

\begin{abstract}
    This paper addresses the topic of equidistribution and recurrence for polynomial sequences over function fields.
	The main focus is to note and correct two small errors in \cite{bl}, contextualized within the broader developing literature on number theory and additive combinatorics in function fields.
    Connected with the resolution of these issues, we also prove new results characterizing intersective polynomials in finite characteristic in terms of various algebraic, combinatorial, and dynamical properties.
\end{abstract}


\section{Introduction}

The goal of this paper is to clarify several facts related to uniform distribution of polynomial sequences and recurrence over function fields.
This paper was spurred by recent advances that have revealed two subtle errors in \cite{bl}:
\begin{enumerate}[(1)]
	\item	A misstatement in the final sentence of \cite[Theorem 0.3]{bl}, erroneously describing a condition on a polynomial $g$ as being both necessary and sufficient for the polynomial sequence $g(n)$ to be well distributed, when the condition is sufficient but not necessary. (The main content of the theorem expressed in the preceding sentences is unaffected by this error.)
	\item	An overly stringent condition used in describing \emph{intersective polynomials} in the remark after \cite[Theorem 9.5]{bl} that leads to clashes with other proposed definitions in the literature based on various equivalent conditions for intersectivity of integer polynomials.
\end{enumerate}
These issues do not affect other results from \cite{bl}.
Nevertheless, recent developments related to (1) and (2) have led to some confusion about the validity of other related results, and finding corrections in the literature has until now required careful reading of footnotes and remarks scattered between different papers (the issue (1) was raised in \cite{cgllw} and addressed in \cite[Remark 2.2]{a}; a description and resolution of issue (2) appears in \cite[p. 3, footnote 3]{ab_fields} and \cite[Remark 4.15]{ab_rings}).
For these reasons, we have written this note to pinpoint exactly where errors occur within the text of \cite{bl} and to settle in a more definitive manner what the necessary corrections are.
We also produce some new results on closely connected topics.
For example, we show that the set of nonzero values of an intersective polynomial (when properly defined) in finite characteristic is a van der Corput set, resolving a conjecture from \cite{llw}.

We elaborate and resolve each issue in turn, addressing (1) in Section \ref{sec: equidistribution} and (2) in Section \ref{sec: intersective}.
In Section \ref{sec: intersective}, we have aimed to provide a rather comprehensive treatment of the topic of intersective polynomials in finite characteristic, and the final subsection (\ref{sec: additional consequences}) includes a proof of equivalences between several properties of algebraic, combinatorial, and dynamical nature that all characterize intersectivity of polynomials in finite characteristic.

For ease of readability, we include in Section \ref{sec: notation} a review of the notation and terminology introduced in \cite{bl}.
Included within Subsection \ref{sec: polynomial sequences} is a small correction to the utilization of \emph{derivational degree} in \cite{bl}.


\section*{Acknowledgements}

EA is supported by the Swiss National Science Foundation under grant TMSGI2-211214.
The present text benefited greatly from the input of Sasha Leibman, whose feedback led to simplifications in Example \ref{eg: finite index subtorus} and the formulations of Theorems \ref{thm: correct iff well distributed} and \ref{thm: correct iff T^c}.
We also thank J\'{e}r\'{e}my Champagne, Zhenchao Ge, Th\'{a}i Ho\`{a}ng L\^{e}, Yu-Ru Liu, and Trevor D. Wooley for comments on an earlier draft that led to expanded remarks related to Theorem \ref{thm: corrected main} and the discussion of derivational degree in Subsection \ref{sec: polynomial sequences}.


\section{Notation and terminology} \label{sec: notation}

In this section, we collect notation and terminology from \cite{bl} that will be needed in the sequel.


\subsection{Basic algebraic objects in finite characteristic}

The goal of \cite{bl} is to establish finite characteristic analogues of classical results of Weyl \cite{weyl} about the distribution of polynomial sequences mod 1.
One of the first steps in formulating analogies is the development of a glossary for translating between the two settings.
In correspondence with the classical algebraic objects $\Z$ (the ring of integers), $\Q$ (the field of rational numbers), and $\R$ (the field of real numbers), we write $\mZ$ for the ring $\F[t]$ of polynomials over a finite field $\F$, $\mQ$ for the field of rational functions $\F(t)$, and $\mR$ for the field of formal Laurent series $\F((t^{-1}))$, which is the completion of $\mQ$ with respect to the metric induced by the absolute value
\begin{equation*}
	\left| \frac{u}{v} \right| = |\F|^{\deg{u} - \deg{v}}.
\end{equation*}
Finally, in analogy with the traditional $\T$ used to denote the compact group $\R/\Z$, we write $\mT$ for the compact quotient group $\mR/\mZ$.

Whereas Weyl's polynomial equidistribution theorem concerns the distribution of values of a polynomial $g(n) \in \R[n]$ taken mod 1 as $n$ runs through the integers, the subject of \cite{bl} is the distribution of values of a polynomial $\mZ$-sequence $g(n) \in \mR[n]$ viewed mod $\mZ$.
(A $\mZ$-sequence is a sequence indexed by $\mZ$.)

The integers $\Z$ and the torus $\T$ have a deep relationship expressed by their duality (in the sense of Pontryagin).
The function $e(x) = e^{2\pi i x}$ usefully displays the duality between $\Z$ and $\T$ via the pairing $\langle n,x \rangle = e(nx)$: this pairing is a group homomorphism in each variable, every character on $\Z$ can be expressed uniquely as $\langle \cdot, x \rangle$ for some $x \in \T$, and every character on $\T$ can be expressed uniquely as $\langle n, \cdot \rangle$ for some $n \in \Z$.
The finite characteristic objects $\mZ$ and $\mT$ possess an analogous duality.
For $x = \sum_{i=-\infty}^r x_i t^i \in \mR$, we write
\begin{equation*}
	e(x) = e^{2\pi i \textup{Tr}_{\F/\F_p}(x_{-1})/p},
\end{equation*}
where $p$ is the characteristic of $\F$ and $\textup{Tr}_{\F/\F_p}$ is the trace map induced by viewing $\F$ as a field extension of its prime subfield $\F_p$.
(If $|\F| = p^k$, the trace can be written explicitly as $\textup{Tr}_{\F/\F_p}(x) = x + x^p + \ldots + x^{p^{k-1}}$.)
Then $\langle n, x \rangle = e(nx)$ represents characters on $\mZ$ and $\mT$ in its different variables in a manner directly analogous to the duality between $\Z$ and $\T$ described above.
While the reader may at this point be distressed by the two distinct meanings of $e(\cdot)$ presented in this paragraph, we assuage this valid concern by noting that $e(\cdot)$ will from now on only be used to denote the function defined on $\mR$ in the finite characteristic setting.


\subsection{Well distribution}

As mentioned above, we are interested in the distribution of values of polynomial sequences in tori $\mT^c$.
To discuss ``distribution'' in precise mathematical terms, we would like to determine the measures arising as limits of certain averages over terms of a sequence $a(n) \in \mT^c$, $n \in \mZ$.
The relevant averaging method comes from the notion of a F{\o}lner sequence.
A \emph{F{\o}lner sequence} in $\mZ$ is a sequence of finite subsets $(\Phi_N)_{N \in \N}$ of $\mZ$ satifsying the asymptotic invariance property
\begin{equation*}
	\lim_{N \to \infty} \frac{\left| (\Phi_N+n) \triangle \Phi_N \right|}{|\Phi_N|} = 0
\end{equation*}
for every $n \in \mZ$.
For example, the sequence of sets $\Phi_N = \{n \in \mZ : \deg{n} \le N\}$ forms a F{\o}lner sequence in $\mZ$.
We say that a $\mZ$-sequence $a(n) \in \mT^c$, $n \in \mZ$, is \emph{well distributed} with respect to a Borel probability measure $\mu$ on $\mT^c$ if for every continuous function $f : \mT^c \to \C$ and every F{\o}lner sequence $(\Phi_N)_{N \in \N}$ in $\mZ$, one has
\begin{equation} \label{eq: well distribution}
	\lim_{N \to \infty} \frac{1}{|\Phi_N|} \sum_{n \in \Phi_N} f(a(n)) = \int_{\mT^c} f~d\mu.
\end{equation}

The limiting measures that arise from polynomial sequences turn out to be supported on finitely many cosets of subgroups of $\mT^c$, and we introduce specific terminology adapted to this situation.
First, if $X \subseteq \mT^c$ is a closed subgroup and $\mu = \mu_X$ is the Haar measure on $X$, we say that a sequence $a(n)$, $n \in \mZ$, satisfying \eqref{eq: well distribution} is well distributed in $X$ rather than directly referring to the measure $\mu$.
More generally, if $X = X_0 + K$ is a finite union of cosets of a closed subgroup $X_0 \subseteq \mT^c$ and there exists $m \in \mZ$ such that for every $r \in \mZ$,
\begin{itemize}
	\item	$a(m\mZ+r) \subseteq X_0 + a(r)$, and
	\item	$a(mn+r)$, $n \in \mZ$, is well distributed in $X_0 + a(r)$,
\end{itemize}
then we say that $a(n)$, $n \in \mZ$, is \emph{well distributed in the components of $X$}.
We note that being well distributed in the components of $X$ is a weaker condition than being well distributed in $X$, as it allows for the different cosets $X_0 + k$, $k \in K$, to be assigned different ``weights.''
In other words, the limiting measure $\mu$ may take the more general form $\mu = \sum_{k \in K} c_k \mu_k$, where $\mu_k$ is the Haar measure on the coset $X_0 + k$ for $k \in K$ and $c_k \in (0,1]$ are rational coefficients with $\sum_{k \in K} c_k = 1$.


\subsection{Polynomial sequences} \label{sec: polynomial sequences}

Let $c \in \N$.
By a \emph{polynomial $\mZ$-sequence in $\mT^c$}, we mean a sequence $g(n)$, $n \in \mZ$, where $g(n) = (g_1(n), \ldots, g_c(n))$ and each $g_i$ is a polynomial with coefficients in $\mR$, reduced mod $\mZ$.
While a polynomial $g(x) \in \mR[x]$ can be represented in its usual form as
\begin{equation} \label{eq: standard representation}
	g(x) = \alpha_0 + \alpha_1 x + \ldots + \alpha_k x^k
\end{equation}
for some $k \in \N$ and coefficients $\alpha_i \in \mR$, this representation is not the most useful one for understanding the distribution of values $g(n)$, $n \in \mZ$.
Instead, we will consider representations making use of \emph{additive} and \emph{separable} polynomials.

A polynomial $\eta(x) \in \mR[x]$ is \emph{additive} if $\eta(x+y) = \eta(x) + \eta(y)$ for every $x,y \in \mR$.
Equivalently, additive polynomials are polynomials that can be expressed as linear combinations of the monomials $x^{p^j}$, $j \ge 0$.
A polynomial $g(x) \in \mR[x]$ is \emph{separable} if it is a linear combination of monomials $x^r$ with $p \nmid r$.
For each $k \in \N$, there is a unique representation $k = p^j r$ with $j \ge 0$ and $p \nmid r$.
Rewriting the monomial $x^k$ as $(x^r)^{p^j}$, we may therefore express a polynomial $g(x) \in \mR[x]$ in the form (uniquely up to rearrangement)
\begin{equation} \label{eq: additive/separable representation}
	g(x) = \alpha_0 + \eta_1(x^{r_1}) + \ldots + \eta_d(x^{r_d})
\end{equation}
for some additive polynomials $\eta_i$ and distinct separable monomials $x^{r_i}$.

In the multidimensional setup, we say that $\eta = (\eta_1, \ldots, \eta_c)$ is additive if each of its coordinates $\eta_i$ is an additive polynomial.
A general polynomial $g = (g_1, \ldots, g_c)$ can then also be represented in the form \eqref{eq: additive/separable representation}.

The phenomenon of additive polynomials points to a curious feature of polynomials in finite characteristic.
Namely, two familiar notions of degree from the zero characteristic setting no longer coincide.
On the one hand, a polynomial $g(x) \in \mR[x]$ has a \emph{degree}, which in the expression \eqref{eq: standard representation} is the number $k$ (assuming $\alpha_k \ne 0$).
On the other hand, $g$ has a \emph{derivational degree} that represents its complexity in terms of successive differences.
Polynomials of degree 0 are the nonzero constant polynomials.
We then define the derivational degree inductively by saying that a polynomial $g$ has derivational degree at most $d$ if for every $y \in \mR$, the polynomial $g(x+y) - g(x)$ has derivational degree at most $d-1$.
In particular, additive polynomials have derivational degree 1.
For a monomials $x^r$, if we expand $r = \sum_{j=0}^s c_j p^j$ in base $p$, then $x^r$ has derivational degree $\sum_{j=0}^s c_j$.
For a polynomial $g(x)$ as written in \eqref{eq: standard representation}, the derivational degree is the maximum of the derivational degrees of those monomials $x^r$ for which the coefficient $\alpha_r$ is nonzero.
Similarly, for a polynomial $g(x)$ as written in \eqref{eq: additive/separable representation}, the derivational degree is the maximum of the derivational degrees of those monomials $x^{r_i}$.

At the beginning of \cite[Section 8]{bl}, the derivational degree is defined in terms of \emph{formal derivatives} rather than successive differences.
However, the description of derivational degree following its introduction in \cite[Section 8]{bl} agrees with the derivational degree as defined above.
Moreover, the proof of \cite[Theorem 8.1]{bl}---which is carried out by induction using differences rather than formal derivatives---is correct once one interprets the derivational degree with the meaning given in the present paper.


\subsection{Subtori}

There is a correspondence between additive polynomial sequences $\eta(n)$ in $\mT^c$ and a certain class of subgroups (which we call \emph{subtori}) of $\mT^c$.
For motivation, we briefly review an analogous situation in the more familiar characteristic zero setting.
Given a linear polynomial $a(n) = (n\alpha_1, \ldots, n\alpha_c) \in \T^c$, $n \in \Z$, the orbit closure $\mO(a) = \overline{a(\Z)}$ is a closed subgroup of $\T^c$.
The subgroup $\mO(a)$ consists of finitely many connected components, each of which is a coset of a subgroup $\mF(a)$ isomorphic to $\T^b$ for some $0 \le b \le c$.
The closed connected subgroup $\mF(a)$ is what we refer to as a \emph{subtorus} of $\T^c$, and the sequence $a(n)$ has the following behavior: if $m \in \N$ denotes the number of connected components of $\mO(a)$, then
\begin{itemize}
	\item	$a(n) \bmod{\mF(a)}$ is $m$-periodic, and
	\item	for each $k \in \{0,1,\ldots,m-1\}$ and $l \in \N$, the sequence $a(lmn+k)$, $n \in \Z$, is dense (in fact, well-distributed) in $\mF(a) + a(k)$.
\end{itemize}
(This phenomenon is described by Weyl in \cite[Section 5]{weyl}, where he also proves a generalization for polynomial sequences.)

In the finite characteristic context, connectedness is no longer a meaningful criterion for distinguishing subtori from other subgroups, since the ``torus'' $\mT$ is a totally disconnected set.
Instead, subtori can be identified either through certain algebraic identities or by means of a decomposition of (additive) polynomial sequences into periodic and totally well-distributed components, as we will describe now.

As an algebraic entity, an \emph{$S$-subtorus} ($S$ for standard) of $\mT^c$ is the image of a torus $\mT^b$, $b \le c$, under a linear mapping $(x_1, \ldots, x_b) \mapsto \sum_{i=1}^b m_i x_i$ with $m_1, \ldots, m_b \in \mZ^c$.
If $\eta(n) = (n\alpha_1, \ldots, n\alpha_c) \in \mT^c$, $n \in \mZ$, is a linear polynomial sequence, then the closure $\mO(\eta) = \overline{\eta(\mZ)}$ is a finite union of cosets of an $S$-subtorus $\mS(\eta)$, and there exists $m \in \mZ$ such that
\begin{itemize}
	\item	$\eta(n) \bmod{\mS(\eta)}$ is $m$-periodic (i.e., $\eta(n+m) - \eta(n) \in \mS(\eta)$ for every $n \in \mZ$), and
	\item	for each $k \in \mZ$ with $\|k\| < \|m\|$ and $l \in \mZ \setminus \{0\}$, the $\mZ$-sequence $\eta(lmn+k)$, $n \in \mZ$, is dense (well-distributed) in $\mS(\eta) + \eta(k)$.
\end{itemize}
(The definition of an $S$-subtorus appears in \cite[Section 1]{bl}, and the relationship between linear polynomials and $S$-subtori is proved in \cite[Section 4]{bl}.)

A more general family of subtori is needed to capture the same type of behavior for additive polynomials of higher degree.
These are the $\Phi$-subtori, named in reference to the \emph{Frobenius endomorphism} $\Phi(x) = x^p$.
A \emph{$\Phi$-subtorus of level $l$} is the image in $\mT^c$ of a torus $\mT^b$ under an additive polynomial mapping $(x_1, \ldots, x_b) \mapsto \sum_{i=1}^b \sum_{j=0}^l m_{i,j} x_i^{p^j}$ with $m_{i,j} \in \mZ^c$.
For additive polynomial sequences $\eta(n) \in \mT^c$, $n \in \mZ$, the closure $\mO(\eta) = \overline{\eta(\mZ)}$ is now a finite union of cosets of a $\Phi$-subtorus $\mF(\eta)$, and there exists $m \in \mZ$ such that
\begin{itemize}
	\item	$\eta(n) \bmod{\mF(\eta)}$ is $m$-periodic, and
	\item	for each $k \in \mZ$ with $\|k\| < \|m\|$ and $l \in \mZ \setminus \{0\}$, the $\mZ$-sequence $\eta(lmn+k)$, $n \in \mZ$, is dense (well-distributed) in $\mF(\eta) + \eta(k)$.
\end{itemize}
(The definition of a $\Phi$-subtorus appears in \cite[Section 1]{bl}, and the connection between additive polynomials and $\Phi$-subtori is described in \cite[Section 7]{bl}.)


\section{Equidistribution theorem} \label{sec: equidistribution}

The main equidistribution theorem, \cite[Theorem 0.3]{bl}, contains a subtle mistake in the last sentence of its formulation.


\subsection{Corrected formulations of the equidistribution theorem}

A corrected formulation of \cite[Theorem 0.3]{bl} is the following.

\begin{theorem} \label{thm: corrected main}
	Any additive polynomial $\mZ$-sequence $\eta(n)$ in $\mT^c$ is well distributed in a set of the form $\mF(\eta) + \eta(K)$, where $\mF(\eta)$ is a $\Phi$-subtorus of level $\le \log_p \deg{\eta}$ of $\mT^c$ and $K$ is a finite subset of $\mZ$.
	For any polynomial $\mZ$-sequence $g(n) = \alpha_0 + \eta_1(n^{r_1}) + \ldots + \eta_d(n^{r_d})$, $n \in \mZ$, where $\alpha_0 \in \mT^c$, $\eta_1, \ldots, \eta_d$ are additive polynomial $\mZ$-sequences and $n^{r_1}, \ldots, n^{r_d}$ are distinct separable monomials, the closure $\mO(g) = \overline{g(\mZ)}$ of $g(\mZ)$ has the form $\mF(g) + g(K)$, where $\mF(g)$ is the $\Phi$-subtorus $\sum_{i=1}^d \mF(\eta_i)$ and $K$ is a finite subset of $\mZ$, and $g(n)$ is well distributed in the components $\mF(g) + g(k)$, $k \in K$, of $\mO(g)$.
	In particular, $g$ is well distributed in $\mT^c$ if $\mT^c = \sum_{i=1}^d \mF(\eta_i)$.
\end{theorem}

\begin{remark}
    As shown in \cite[Theorem 8.1]{bl}, the set $K$ can be chosen as $K = \{k \in \mZ : \|k\| < \|m\|\}$ for some $m \in \mZ \setminus \{0\}$.
    Well distribution in the components of $\mO(g) = \mF(g) + g(K)$ then means that for each $k \in K$, the $\mZ$-sequence $g(mn+k)$, $n \in \mZ$, is well distributed in $\mF(g)+g(k)$.
\end{remark}

The reader may be readily forgiven if unable to detect the difference between \cite[Theorem 0.3]{bl} and Theorem \ref{thm: corrected main}; all that has changed is the removal of a second ``f'' to modify ``iff'' into ``if'' in the final sentence of the theorem.
Nevertheless, this pesky extra ``f'' has significant consequences that have become relevant in recent work further developing the theory of uniform distribution in the finite characteristic setting (\cite{a, cgllw}).
In \cite{bl}, what is actually proved in the text is all but the final sentence of Theorem \ref{thm: corrected main}, with the implication that the final sentence follows by a straightforward (but unspecified) argument.
Let us observe that the final sentence of Theorem \ref{thm: corrected main} as written above does easily follow from the preceding portions of the theorem: indeed, if $\mT^c = \sum_{i=1}^d \mF(\eta_i)$, then each of the components $\mF(g) + g(k)$ is equal to $\mT^c$, so $g$ is well distributed in $\mT^c$.
On the other hand, the erroneous ``only if'' direction does not follow from the preceding phrases, as it is within the realm of possibility that $\mF(g)$ is a subgroup of finite index in $\mT^c$ and $g(k)$, $k \in K$, enumerates the different cosets.
The following examples demonstrates precisely this phenomenon.

\begin{example} \label{eg: finite index subtorus}
    We adapt the third example appearing right before \cite[Theorem 0.3]{bl} to produce an additive polynomial $\eta$ such that $\eta(n)$, $n \in \mZ$, is well distributed in $\mT$ but $\mF(\eta)$ is a proper subgroup of $\mT$.
    
    Let $\alpha \in \mT$ be irrational, and consider the additive polynomial $\eta_0(n) = \alpha n - t^{q-1} \alpha^q n^q$, where $q = |\F|$.
	(The corresponding polynomial when $\F = \F_2$ is labeled $g_6(n)$ in the example prior to \cite[Theorem 0.3]{bl}.)
	Since $\alpha$ is irrational, the $\Phi$-subtorus $\mF(\eta_0)$ is the subgroup
	\begin{multline*}
		\mF(\eta_0) = \{x - t^{q-1}x^q : x \in \mT\}
        = \left\{ \sum_{i=1}^{\infty} u_i t^{-i} - \sum_{i=1}^{\infty} u_i t^{-qi+q-1} : u_i \in \F \right\} \\
		 = \left\{ (u_1-u_1)t^{-1} + u_2 t^{-2} + \ldots + u_q t^{-q} + (u_{q+1} - u_2) t^{-(q+1)} + \ldots : u_i \in \F \right\} \\
		 = \left\{ \sum_{i=1}^{\infty} u_i t^{-i} : u_i \in \F, u_1 = 0 \right\}
         = \left\{ \sum_{i=2}^{\infty} u_i t^{-i} : u_i \in \F \right\}
	\end{multline*}
	of index $q$.
	(This is closely related to the $\Phi$-subtorus $\mF_4$ appearing in the example in \cite[p. 934]{bl}.)
	We now take $\eta(n) = \eta_0(n) + \frac{n}{t}$.
    The $\Phi$-subtorus $\mF(\eta)$ is still the proper subgroup $\mF(\eta_0)$.
    To see this, we use the observation from \cite[Section 7]{bl} that if $\eta'(n) = \eta(mn)$ for some $m \in \mZ \setminus \{0\}$, then $\mF(\eta') = \mF(\eta)$.
	Taking $m = t$ and noting that $\eta'(n) = \eta(tn) = \eta_0(tn)$ for every $n \in \mZ$, we have $\mF(\eta) = \mF(\eta') = \mF(\eta_0') = \mF(\eta_0)$.
    However, the term $\frac{n}{t}$ cycles through the values $ut^{-1}$, $u \in \F$, corresponding to cosets of $\mF(\eta)$.
    In the notation of Theorem \ref{thm: corrected main}, we may take $K = \F$, and $\eta(tn+u)$ is well distributed in $\mF(\eta) + \eta(u) = \mF(\eta) + ut^{-1}$ for each $u \in K = \F$.
    In this case, the well distribution in components provided by Theorem \ref{thm: corrected main} means that $\eta$ is well distributed in $\mT$, since each coset of $\mF(\eta)$ has exactly one representative in $K$.
\end{example}

The following statement provides a corrected ``if and only if'' characterization of when a polynomial sequence $g(n)$, $n \in \mZ$, is well distributed in $\mT^c$.

\begin{theorem} \label{thm: correct iff well distributed}
    A polynomial $\mZ$-sequence $g(n)$, $n \in \mZ$, is well distributed in $\mT^c$ if and only if either $\mF(g) = \mT^c$ or $\mF(g)$ is a subgroup of finite index in $\mT^c$ and $g$ is well distributed in the finite group $\mT^c/\mF(g)$.
\end{theorem}

\begin{proof}
    Before proving the theorem, we note that if $\mF(g) = \mT^c$, then $\mF(g)$ is a subgroup of finite index (index $1$), and the quotient group $\mT^c/\mF(g)$ is the trivial group, so $g$ is automatically well distributed in $\mT^c/\mF(g)$.
    Thus, we do not have to deal with the case $\mF(g) = \mT^c$ separately.
    
    Now we proceed with the proof.
    Suppose $g(n)$, $n \in \mZ$, is well distributed in $\mT^c$.
    By Theorem \ref{thm: corrected main}, the closure $\mO(g)$ is of the form $\mF(g) + g(K)$ for some finite set $K \subseteq \mZ$, but being well distributed implies $\mO(g) = \mT^c$, so
    \begin{equation*}
        \mT^c = \mF(g) + g(K).
    \end{equation*}
    Therefore, the $\Phi$-subtorus $\mF(g)$ has index at most $|g(K)| \le |K| < \infty$.
    By assumption, $g$ is well distributed in $\mT^c$, so it is also well distributed in every compact quotient group, since the Haar measure on $\mT^c$ will project to the Haar measure on the quotient.
    In particular, $g$ is well distributed in $\mT^c/\mF(g)$.

    Conversely, suppose $\mF(g)$ is a subgroup of finite index in $\mT^c$ and $g$ is well distributed in $\mT^c/\mF(g)$.
    By \cite[Theorem 8.1]{bl}, let $m \in \mZ \setminus \{0\}$ such that for each $k \in K = \{l \in \mZ : \|l\| < \|m\|\}$, the $\mZ$-sequence $g(mn+k)$, $n \in \mZ$, is well distributed in $\mF(g) + g(k)$.
    Then $g(n) \equiv g(k) \pmod{\mF(g)}$ for $n \equiv k \pmod{m}$.
    Since $g(n)$, $n \in \mZ$, is well distributed in $\mT^c/\mF(g)$, we deduce that $g(k)$, $k \in K$, enumerates each element of $\mT^c/\mF(g)$ an equal number of times.
    Therefore, $g$ is well distributed with respect to the measure
    \begin{equation*}
        \frac{1}{|K|} \sum_{k \in K} \mu_{\mF(g) + g(k)} = \frac{1}{[\mT^c : \mF(g)]} \sum_{x \in \mT^c/\mF(g)} \mu_{\mF(g)+x} = \mu_{\mT^c}.
    \end{equation*}
    That is, $g$ is well distributed in $\mT^c$.
\end{proof}

Going in the other direction, the condition $\mF(g) = \mT^c$ characterizes a stronger equidistributional property of polynomials, as expressed in the next theorem.

\begin{theorem} \label{thm: correct iff T^c}
    The following are equivalent for a polynomial $\mZ$-sequence $g(n)$, $n \in \mZ$, in $\mT^c$:
    \begin{enumerate}[(i)]
        \item $\mF(g) = \mT^c$
        \item for any $m \in \mZ \setminus \{0\}$ and $k \in \mZ$, the $\mZ$-sequence $g(mn+k)$, $n \in \mZ$, is well distributed in $\mT^c$.
    \end{enumerate}
\end{theorem}

\begin{proof}
    Suppose (i) holds.
    The key observation comes from the remark in \cite[p. 946]{bl}: if $m \in \mZ \setminus \{0\}$ and $k \in \mZ$, then the polynomial $g'_{m,k}(n) = g(mn+k)$ satisfies $\mF(g'_{m,k}) = \mF(g)$.
    Hence, for each $m \in \mZ \setminus \{0\}$ and $k \in \mZ$, we have $\mF(g'_{m,k}) = \mT^c$, so $g'_{m,k}$ is well distributed in $\mT^c$ by Theorem \ref{thm: corrected main}.

    Conversely, suppose (ii) holds.
    By \cite[Theorem 8.1]{bl}, we may pick $m \in \mZ \setminus \{0\}$ so that $g(mn+k)$, $n \in \mZ$, is well distributed in $\mF(g) + g(k)$ for each $k \in \mZ$ with $\|k\| < \|m\|$.
    Condition (ii) then implies $\mF(g) + g(k) = \mT^c$.
    But $g(k)$ is a single element, so we conclude that $\mF(g) = \mT^c - g(k) = \mT^c$.
\end{proof}


\subsection{A note on the counterexample of Champagne--Ge--L\^{e}--Liu--Wooley}

The error in the original statement of \cite[Theorem 0.3]{bl} was brought to light by a counterexample in \cite{cgllw} to a statement dubbed ``Hypothesis 5.1,'' which can be deduced from the ``iff'' statement in \cite[Theorem 0.3]{bl}.
The following is a restatement of ``Hypothesis 5.1'' in the language and notation set in Section \ref{sec: notation}.

\begin{premise} \label{premise}
	Let $g(n) = \alpha_0 + \eta_1(n^{r_1}) + \ldots + \eta_d(n^{r_d})$, $n \in \mZ$, where $\alpha_0 \in \mT^c$, $\eta_1, \ldots, \eta_d$ are additive polynomial $\mZ$-sequences and $n^{r_1}, \ldots, n^{r_d}$ are distinct separable monomials.
	If $\eta_i(n)$ is well distributed in $\mT^c$ for some $i \in \{1, \ldots, d\}$, then $g(n)$ is also well distributed in $\mT^c$.
\end{premise}

We now give a streamlined construction of a counterexample to Premise \ref{premise}.
Utilizing Example \ref{eg: finite index subtorus} above of an additive polynomial whose associated $\Phi$-subtorus is of finite index, we will exhibit an explicit polynomial $g$ that contradicts Premise \ref{premise}.
The behavior of this polynomial is similar to the counterexample in \cite{cgllw}, but the argument in \cite{cgllw} is less direct, relying on a general statement (\cite[Theorem 5.3]{cgllw}) about the existence of additive polynomials having a certain relationship with group characters on $\mT$.

\begin{proposition}
	There exists an additive polynomial $\mZ$-sequence $\eta(n)$ in $\mT$ and a number $r \in \N$, $p \nmid r$, such that $\eta(n)$ is well distributed in $\mT$ but $g(n) = \eta(n^r)$ is not well distributed in $\mT$.
\end{proposition}

\begin{proof}
	For technical reasons, we consider separately the cases $\F = \F_2$ and $\F \ne \F_2$. \\
	
	\underline{Case 1}: $\F \ne \F_2$.
	
	Let $\eta$ be the additive polynomial defined in Example \ref{eg: finite index subtorus}.
	For $r \in \N$, the polynomial $\mZ$-sequence $g(n) = \eta(n^r)$ is well distributed in the components $\mF(\eta) + g(u)$, $u \in \F$, of $\mO(g) = \mF(\eta) + g(\F)$ by Theorem \ref{thm: corrected main}.
	Hence, all that we need to do to ensure that $g(n)$ is not well distributed in $\mT$ is to pick $r$ such that $\{u^r : u \in \F\}$ is a proper subset of $\F$.
	For example, we can take $r = q-1$ so that $\{u^r : u \in \F\} = \{0,1\}$. \\
	
	\underline{Case 2}: $\F = \F_2$.
	
	We follow a similar strategy as in the previous case but with a small modification so that the group $\mF(\eta)$ has index 4 rather than index $|\F| = 2$.
	
	Let $\alpha \in \mT$ be irrational, and consider the additive polynomial $\eta_0(n) = \alpha n + (t+t^2) \alpha^2 n^2$.
	Since $\alpha$ is irrational, the $\Phi$-subtorus $\mF(\eta_0)$ is the subgroup
	\begin{multline*}
		\mF(\eta_0) = \{x + (t+t^2)x^2 : x \in \mT\} 
        = \left\{ \sum_{i=1}^{\infty} u_i t^{-i} + \sum_{i=1}^{\infty} u_i t^{-2i+1} + \sum_{i=1}^{\infty} u_i t^{-2i+2} : u_i \in \F_2 \right\} \\
		 = \left\{ (u_1+u_1)t^{-1} + (u_2+u_2)t^{-2} + (u_3+u_2)t^{-3} + (u_4+u_3)t^{-4} + (u_5+u_3)t^{-5} + \ldots : u_i \in \F_2 \right\} \\
		 = \left\{ \sum_{i=1}^{\infty} u_i t^{-i} : u_i \in \F_2, u_1 = u_2 = 0 \right\} = \left\{ \sum_{i=3}^{\infty} u_i t^{-i} : u_i \in \F_2 \right\}.
	\end{multline*}
	Let $\eta(n) = \eta_0(n) + \frac{n}{t^2}$.
	As in Example \ref{eg: finite index subtorus}, we have $\mF(\eta) = \mF(\eta_0)$ but $\eta(n)$ is well distributed in $\mF(\eta) + \eta(\F_2 + \F_2 t) = \mT$.
	
	We now consider the polynomial $g(n) = \eta(n^3)$.
	Given $n = \sum_{i=0}^N v_i t^i \in \mZ$, we have
	\begin{equation*}
		\frac{n^3}{t^2} \equiv \frac{(v_0 + v_1t)^3}{t^2} \equiv \frac{v_0 + v_0v_1t + v_0v_1t^2 + v_1t^3}{t^2} \equiv v_0v_1 t^{-1} + v_0 t^{-2} \pmod{\mZ}.
	\end{equation*}
	Thus, as $n$ ranges over $\mZ$, the polynomial $\frac{n^3}{t^2}$ cycles through the values
	\begin{equation*}
		0 \mapsto 0, \quad 1 \mapsto t^{-1}, \quad t \mapsto 0, \quad \text{and} \quad 1+t \mapsto t^{-1} + t^{-2}
	\end{equation*}
	mod $\mZ$.
	By Theorem \ref{thm: corrected main}, we conclude that $g(n)$ is well-distributed in the components of $\mF(\eta) + \{0, t^{-1}, 0, t^{-1} + t^{-2}\}$, which misses the coset $\mF(\eta) + \{t^{-2}\}$.
\end{proof}


\subsection{Impact on subsequent papers}

In the past few years, \cite[Theorem 0.3]{bl} has been quoted in several follow-up works, so it is worth commenting briefly on the usage of \cite[Theorem 0.3]{bl} in other papers.
The abbreviated summary of our more detailed discussion to follow is that there is no cause for alarm and the extraneous ``f'' appearing in \cite{bl} can safely be excised as in Theorem \ref{thm: corrected main} above without consequential downstream effects.

The issue with the original formulation of \cite[Theorem 0.3]{bl} was raised in \cite{cgllw} following its usage in \cite{a} to prove a function field analogue of Weyl's classical polynomial equidistribution theorem.
As explained in \cite[Remark 2.2]{a}, the argument in \cite{a} uses only the correct ``if'' direction, and the latest version of the article uses the updated formulation of Theorem \ref{thm: corrected main} appearing above rather than the original formulation from \cite{bl}.

Versions of \cite[Theorem 0.3]{bl} were previously quoted in \cite[Theorem 1.7]{ab_low_characteristic} and \cite[Theorem 5.2]{ab_fields}.
The first of these, \cite[Theorem 1.7]{ab_low_characteristic}, includes only correct portions of the theorem, omitting the final sentence entirely.
The second, \cite[Theorem 5.2]{ab_fields}, quoted the incorrect ``if and only if'' statement in a previous version.
This has now been updated to only include the correct ``if'' direction, and subsequent arguments in the proof of \cite[Theorem 1.5]{ab_fields}---the main theorem deduced from \cite[Theorem 0.3]{bl}---have been repaired accordingly in \cite[Section 5]{ab_fields}.


\section{Intersective polynomials} \label{sec: intersective}

In this section we address the issue described in item (2) of the introduction, which arises from the remark following \cite[Theorem 9.5]{bl} concerning the meaning of \emph{intersective} polynomials in the function field setting.
In order to enable our discussion, let us begin by labeling various properties of a nonzero polynomial $q(n) \in \mZ[n]$ that have appeared under the moniker ``intersective:''
\begin{enumerate}[(P1)]
	\item	for any subgroup $\Lambda$ of finite index in $\mZ$, there exists $m \in \mZ$ such that $q(mn) \in \Lambda$ for all $n \in \mZ$
	\item	for any subgroup $\Lambda$ of finite index in $\mZ$, there exists $m \in \mZ$ such that $q(m) \in \Lambda$
	\item	$q$ has a root mod $m$ for every $m \in \mZ \setminus \{0\}$
	\item	for any subset $E \subseteq \mZ$ with positive density, there exist $x,y \in E$ and $m \in \mZ$ such that $x - y = q(m)$
\end{enumerate}

There are several choices of ``density'' we may take in (P4) that all lead to an equivalent notion, but for concreteness, we will use the \emph{upper Banach density}
\begin{equation*}
    d^*(E) = \sup_{(\Phi_N)} \limsup_{N \to \infty} \frac{|E \cap \Phi_N|}{|\Phi_N|},
\end{equation*}
where the supremum is taken over all F{\o}lner sequences $(\Phi_N)_{N \in \N}$ in $\mZ$.

We briefly summarize the usage of these different properties for defining intersective polynomials over $\mZ$ in the literature.
Intersective polynomials are defined in terms of (P4) in \cite{le}, where it is also posed as an open question whether (P4) is equivalent to (P3) (see \cite[Problem 7]{le}).
Polynomials satisfying property (P1) are termed intersective in \cite{bl}, with the suggestion that (P1) coincides with (P4) (see the remark following \cite[Theorem 9.5]{bl}).
L\^{e}'s definition using (P4) as well as the conjectural view of an equivalence between (P4) and (P3) was propogated in \cite{li} (see \cite[Definition 1]{li}), where (P4) is shown to imply a quantitatively refined version of itself.
The property (P2) is used as the definition of intersective in \cite{mishra} (see \cite[Eq. (1.1)]{mishra}), based in part on communication with the first and second authors of this note.
Motivated by corresponding notions in the integer setting, the property (P3) is used as a definition in \cite{llw} (see the discussion following \cite[Theorem 1.9]{llw}).
In \cite{llw}, the authors also observe that the definition via (P3) is not obviously the same as the definition in \cite{bl} and formulate a question aimed at determining whether the properties are equivalent.
As written, \cite[Question 1]{llw} asks whether (P2) and (P3) are equivalent, as the reinterpretation of the results from \cite[Section 9]{bl} attributed as the Bergelson--Leibman theorem in \cite[Section 6.2]{llw} settles on using (P2) as a definition and refer to the use of the stronger condition (P1) in \cite{bl} as a ``misprint'' (see \cite[p. 32, footnote 2]{llw}).
(The discussion below seeks to clarify this situation and more fully address the ``misprint.'')

While the issues of determining implications and (in)equivalences between the different properties above has been dealt with by now in work of the first and second authors (see \cite[p. 3, footnote 3]{ab_fields} and \cite[Theorem 4.13, Definition 4.14, and Remark 4.15]{ab_rings}), this seems not to have fully settled the matter.
Indeed, even after the appearance of an early version of \cite{ab_fields}, the notion of intersectivity was described as ``being in a state of flux'' in \cite[p. 32, footnote 3]{llw}, in part due to the lack of a clear proof and exposition in the existing literature.

The goal of this section is to rectify the current lack of clarity.
We will show that (P2), (P3), and (P4) are equivalent, and any of these three properties is suitable for defining intersectivity.
On the other hand, the property (P1) is strictly stronger than the other properties and conflicts with the usual meaning of the word ``intersective'' as used in other contexts within additive combinatorics.
In the final subsection, we prove a rich assortment of properties that intersective polynomials exhibit vis-\`{a}-vis additive combinatorics and dynamical recurrence.

\subsection{Property (P1)}

As a first observation, we note that (P1) immediately implies (P2) by taking $n = 1$.
We now show that this implication does not go the other way; that is, there are polynomials $q(n) \in \mZ[n]$ satisfying (P2) but not (P1).

\begin{proposition} \label{prop: inequivalence}
	There exists a polynomial $q(n) \in \mZ[n]$ satisfying (P2) but not (P1).
\end{proposition}

\begin{proof}
	We use an example described in \cite[p. 3, footnote 3]{ab_fields} and repeated in \cite[Remark 4.15]{ab_rings}.
	Let $q(n) = n+1$.
	Then $q$ satisfies (P2): given a finite index subgroup $\Lambda$, we take any $m \in \Lambda - 1$ and have $q(m) \in \Lambda$.
	However, $q$ does not satisfy (P1): given any $m \in \mZ$, if we take $n = t$, then $q(mn) = q(tm) \equiv 1 \pmod{t}$, so the finite index subgroup $\Lambda = t\mZ$ witnesses the failure of property (P1).
\end{proof}

In fact, we can give a simple characterization of polynomials satisfying (P1).

\begin{proposition}
    A polynomial $q(n) \in \mZ[n]$ satisfies (P1) if and only if $q(0) = 0$.
\end{proposition}

\begin{proof}
    Suppose $q(0) = 0$, and let $\Lambda \le \mZ$ be a subgroup of finite index.
    Then taking $m = 0 \in \mZ$, we have $q(mn) = q(0) = 0 \in \Lambda$ for every $n \in \mZ$, so $q$ satisfies (P1).

    Conversely, suppose (P1) holds.
    Let $\Lambda \le \mZ$ be a subgroup of finite index.
    By (P1), there exists $m \in \mZ$ such that $q(mn) \in \Lambda$ for all $n \in \mZ$, in particular for $n = 0$.
    Thus $q(0) \in \Lambda$.
    But $\Lambda$ was an arbitrary subgroup of finite index, so we conclude that $q(0) = 0$.
\end{proof}

\subsection{Proof of equivalences}

Now we turn to the proof of equivalences between (P2), (P3), and (P4).
All of these equivalences can be extracted from \cite[Theorem 4.13]{ab_rings}, which deals with a more general setting.
For ease of reference, we include here a simplified proof that is available in the setting at hand using only results and methods from \cite{bl}.

\begin{theorem} \label{thm: equivalences}
	Properties (P2), (P3), and (P4) are all equivalent for polynomials $q(n) \in \mZ[n]$.
\end{theorem}

\begin{proof}
	(P2)$\implies$(P3).
	For this implication, it suffices to note that $m\mZ$ is a subgroup of finite index in $\mZ$ for every $m \in \mZ \setminus \{0\}$. \\
	
	(P3)$\implies$(P4).
	We use the method of proof from \cite[Section 9]{bl}.
	Let $E \subseteq \mZ$ be a set of positive density.
	By the Furstenberg correspondence principle (see, e.g., \cite[Theorem 4.17]{etdp}), there exists a probability space $(X, \mu)$, a measurable subset $A \subseteq X$ with $\mu(A) > 0$, and an action of $(\mZ,+)$ by measure-preserving transformations $T(n) : X \to X$, $n \in \mZ$, such that for every $n \in \mZ$, the set $E \cap (E - n)$ has density bounded from below by $\mu(A \cap T(-n)A)$.
	Thus, in order to establish (P4), it suffices to prove $\mu(A \cap T(-q(m))A) > 0$ for some $m \in \mZ$.
	We will prove the stronger claim that for every $\eps > 0$, there exists $k, m \in \mZ$ such that for any F{\o}lner sequence $(\Phi_N)_{N \in \N}$ in $\mZ$, one has
	\begin{equation*}
		\lim_{N \to \infty} \frac{1}{|\Phi_N|} \sum_{n \in \Phi_N} \mu(A \cap T(-q(mn+k)A)) > \mu(A)^2 - \eps.
	\end{equation*}
	This ergodic-theoretic claim bears a strong resemblance to \cite[Theorem 9.3]{bl}, and we will prove it in exactly the same way.
	
    Put $f = 1_A$, and let $\lambda$ be the spectral measure of $f$ so that the closure of $T(\mZ)f$ in $L^2(\mu)$ can be replaced by $L^2(\mT,\lambda)$ with the operators $T(w)$, $w \in \mZ$, being represented by multiplication by $e(w \cdot x)$ and $f$ being represented by the constant function $1$.
	Then $\tilde{f} = \lim_{N \to \infty} \frac{1}{|\Phi_N|} \sum_{n \in \Phi_N} T(n)f$ is represented by $1_{\{0\}} \in L^2(\mT,\lambda)$.
	Hence, by the mean ergodic theorem, $\lambda(\{0\}) = \|\tilde{f}\|_{L^2(\mu)}^2 \ge \mu(A)^2$.
	
	Now, $\mu(A \cap T(-w)A) = \int_{\mT} e(w \cdot x)~d\lambda(x)$.
	Fix $x \in \mT$.
	By \cite[Theorem 8.1]{bl}, there exists $m_x \in \mZ$ such that if $m \in \mZ$ is divisible by $m_x$, then the $\mZ$-sequence $q(mn+k) \cdot x$, $n \in \mZ$, is well distributed in a translated $\Phi$-subtorus $\mF_x + q(k) \cdot x \subseteq \mT$ for every $k \in \mZ$ with $\|k\| < \|m\|$.
	Using property (P3), there exists $k \in \mZ$ such that $q(k) \equiv 0 \pmod{m}$.
	Since $m\mZ \le \mZ$ is an ideal, the polynomial $q$ is well-defined mod $m$.
	In particular, we may choose $k$ such that $\|k\| < \|m\|$ and $q(mn+k) \equiv 0 \pmod{m}$ for every $n \in \mZ$.
	For this choice of $k$, the translated $\Phi$-subtorus $\mF_x + q(k) \cdot x$ coincides with the (untranslated) $\Phi$-subtorus $\mF_x$.
	Hence, the $\mZ$-sequence $e(q(mn+k) \cdot x)$, $n \in \mZ$, is well distributed in the subgroup $e(\mF_x)$ of the group $P$ of $p$th roots of unity in $\C$, so
	\begin{equation*}
		\lim_{N \to \infty} \frac{1}{|\Phi_N|} \sum_{n \in \Phi_N} e(q(mn+k) \cdot x) = \begin{cases}
			0, & \text{if}~e(\mF_x) = P; \\
			1, & \text{if}~e(\mF_x) = \{1\}.
		\end{cases}
	\end{equation*}
	
	Take $m \in \mZ$ sufficiently divisible (e.g., the product of all elements of degree at most $N$ for some large $N$) so that the set $D = \{x \in \mT : m_x \mid m\} \cup \{0\}$ satisfies $\lambda(\mT \setminus D) < \eps$, and let $k \in \mZ$ with $\|k\| < \|m\|$ such that $q(k) \equiv 0 \pmod{m}$.
	Then decomposing $\mT = \{0\} \cup (D \setminus \{0\}) \cup (\mT \setminus D)$ and applying the bounded convergence theorem, we may compute
	\begin{multline*}
		\lim_{N \to \infty} \frac{1}{|\Phi_N|} \sum_{n \in \Phi_N} \mu(A \cap T(-q(mn+k)A)) = \lim_{N \to \infty} \frac{1}{|\Phi_N|} \sum_{n \in \Phi_N} \int_{\mT} e(q(mn+k) \cdot x)~d\lambda(x) \\
		 = \int_{\mT} \lim_{N \to \infty} \frac{1}{|\Phi_N|} \sum_{n \in \Phi_N} e(q(mn+k) \cdot x)~d\lambda(x) \\
		 = \int_{\{0\}} \underbrace{\lim_{N \to \infty} \frac{1}{|\Phi_N|} \sum_{n \in \Phi_N} e(q(mn+k) \cdot x)~d\lambda(x)}_{=1} \\
		 + \int_{D \setminus \{0\}} \underbrace{\lim_{N \to \infty} \frac{1}{|\Phi_N|} \sum_{n \in \Phi_N} e(q(mn+k) \cdot x)~d\lambda(x)}_{\ge 0} \\
		 + \int_{\mT \setminus D} \underbrace{\lim_{N \to \infty} \frac{1}{|\Phi_N|} \sum_{n \in \Phi_N} e(q(mn+k) \cdot x)~d\lambda(x)}_{1\text{-bounded}} \\
		 \ge \lambda(\{0\}) + 0 - \lambda(\mT \setminus D) > \mu(A)^2 - \eps.
	\end{multline*}
	This proves the claim, so $q$ satisfies property (P4). \\
	
	(P4)$\implies$(P2).
	Let $\Lambda$ be a subgroup of finite index in $\mZ$.
	Then $\Lambda$ has density $\frac{1}{[\mZ : \Lambda]} > 0$, so by (P4), there exist $x,y \in \Lambda$ and $m \in \mZ$ such that $x - y = q(m)$.
	But $x - y \in \Lambda$, so we have $q(m) \in \Lambda$.
	That is, $q$ satisfies (P2).
\end{proof}

Combining Proposition \ref{prop: inequivalence} and Theorem \ref{thm: equivalences}, we see that the suggested equivalence between (P1) and (P4) in the remark following \cite[Theorem 9.5]{bl} does not hold.
The intent in \cite{bl} was to define a notion of intersectivity capturing the combinatorial property (P4) in terms of algebraic properties of polynomials, motivated by the the connection between (P4)  and the broader notion of \emph{intersective sets} in additive combinatorics \cite{ruzsa} (or equivalently, sets of \emph{measurable recurrence} in ergodic theory; see Subsection \ref{sec: additional consequences} below).
The property (P1) unfortunately falters in this mission, but both (P2) and (P3) serve as suitable replacements for defining the notion of intersectivity for polynomials $q(n) \in \mZ[n]$.

\subsection{Additional consequences of intersectivity} \label{sec: additional consequences}

The proof of the implication (P3)$\implies$(P4) reveals that intersective polynomials have significantly stronger combinatorial and recurrence properties than one may anticipate immediately from the formulation of (P4).
We enumerate several additional consequences in this subsection.
First, we introduce the relevant terminology.

\begin{definition} \label{definitions}
    Let $S \subseteq \mZ \setminus \{0\}$.
    We say that $S$ is
    \begin{itemize}
        \item \emph{syndetic} if there exists a finite set $F \subseteq \mZ$ such that $S + F = \{s+f : s \in S, f \in F\} = \mZ$;
        \item \emph{intersecitve} if for every set $E \subseteq \mZ$ of positive density, $(E - E) \cap S \ne \es$;
        \item \emph{chromatically intersective} if for every finite partition $\mZ = \bigcup_{i=1}^r C_i$, there exists $i \in \{1,\ldots, r\}$ such that $(C_i - C_i) \cap S \ne \es$;
        \item a \emph{set of measurable recurrence} if for every action of $\mZ$ by measure-preserving transformations $T(n)$, $n \in \mZ$, on a probability space $(X, \mu)$ and every measurable set $A \subseteq X$ with $\mu(A) > 0$, there exists $s \in S$ such that $\mu(A \cap T(-s)A) > 0$;
        \item a \emph{set of topological recurrence} if for every action of $\mZ$ by homeomorphisms $T(n)$, $n \in \mZ$, on a compact metric space $X$ and every nonempty open set $U \subseteq X$, there exists $s \in S$ such that $U \cap T(-s)U \ne \es$;
        \item a \emph{set of operatorial recurrence} if for every unitary action $U(n)$, $n \in \mZ$, of $\mZ$ on a Hilbert space $\mH$ and every $f \in \mH$ with $Pf \ne 0$, where $P$ is the orthogonal projection onto the space of $U$-invariant elements, there exists $s \in S$ such that $\left\langle U(s)f, f \right\rangle \ne 0$;
        \item a \emph{van der Corput set} if for every F{\o}lner sequence $(\Phi_N)_{N \in \N}$ in $\mZ$, if $u(n)$, $n \in \mZ$, is a bounded sequence of complex numbers and
        \begin{equation*}
            \lim_{N \to \infty} \frac{1}{|\Phi_N|} \sum_{n \in \Phi_N} u(n+s) \overline{u(n)} = 0
        \end{equation*}
        for every $s \in S$, then
        \begin{equation*}
            \lim_{N \to \infty} \frac{1}{|\Phi_N|} \sum_{n \in \Phi_N} u(n) = 0.
        \end{equation*}
    \end{itemize}
\end{definition}

Let us describe the relations between the terms defined in Definition \ref{definitions}.

As shown in \cite[Theorem 2.2]{bm}, intersective sets and sets of measurable recurrence are equivalent notions.
Similarly, chromatically intersective sets are the same as sets of topological recurrence (see \cite[Theorem 2.6]{bm}).
Given a finite partition $\mZ = \bigcup_{i=1}^r C_i$, at least one of the sets $C_i$ must be of positive density, so every intersective set (set of measurable recurrence) is chromatically intersective (a set of topological recurrence).

It was shown in \cite{ftd, saul} that the notion of a van der Corput set does not depend on the choice of F{\o}lner sequence $(\Phi_N)_{N \in \N}$.
Moreover, by \cite[Theorem 3.5]{ftd}, a set $S \subseteq \mZ \setminus \{0\}$ is a van der Corput set if and only if it is a set of operatorial recurrence (in the context of $\Z$-actions, this equivalence was earlier demonstrated in \cite[Theorem 1.1]{peres} and rediscovered in \cite[Theorem 2]{nrs}).
By considering the unitary operators on $L^2(\mu)$ induced by a measure-preserving action $T(n)$, $n \in \mZ$, it is easy to see that every set of operatorial recurrence is a set of measurable recurrence.

A summary of the relations between notions defined in Definition \ref{definitions} is provided by the following diagram.

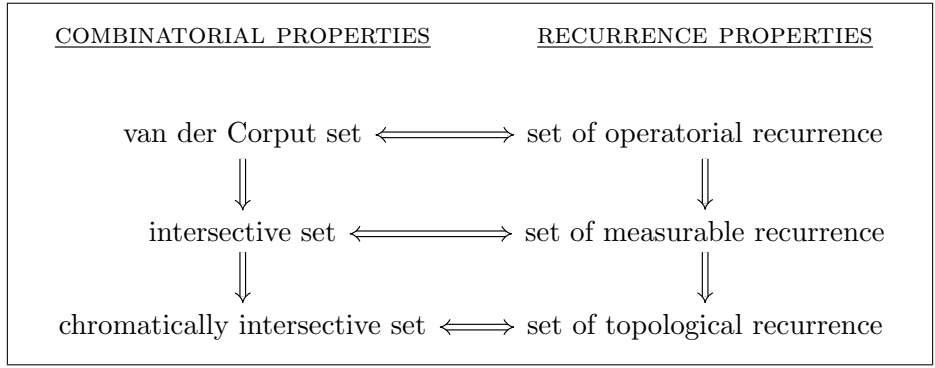
\begin{figure}[h]
    \begin{mdframed}[align = center, userdefinedwidth=32em]
    \begin{center}
	   \begin{tikzcd}
            \underline{\textsc{combinatorial properties}} & \underline{\textsc{recurrence properties}} \\
            \text{van der Corput set} \arrow[r, Leftrightarrow] \arrow[d, Rightarrow] & \text{set of operatorial recurrence} \arrow[d, Rightarrow] \\
            \text{intersective set} \arrow[r, Leftrightarrow] \arrow[d, Rightarrow] & \text{set of measurable recurrence} \arrow[d, Rightarrow] \\
            \text{chromatically intersective set} \arrow[r, Leftrightarrow] & \text{set of topological recurrence}
        \end{tikzcd}
    \end{center}
    \end{mdframed}
    \caption{Relations between combinatorial and recurrence properties in Definition \ref{definitions}.}
    \label{figure}
\end{figure}

The same implications hold for the corresponding notions in the integer setting.
Moreover, for subsets of $\Z$, the unidirectional implication arrows cannot be reversed: K\v{r}\'{i}\v{z} \cite{kriz} constructed an example of a set that is chromatically intersective but not intersective, and Bourgain \cite{bourgain} constructed an example of an intersective set that is not a van der Corput set.
In the function field context, the construction of K\v{r}\'{i}\v{z} was adapted by Forrest \cite[Theorem 3.7]{forrest} to show that there is a chromatically intersective subset of $\F_2[t]$ that is not intersective.
It is an interesting open problem whether K\v{r}\'{i}\v{z}'s result extends to function fields in odd characteristic.
There is also no known example of an intersective subset of $\mZ$ that is not a van der Corput set (see \cite[Problem 9]{le}). \\

The next theorem summarizes enhanced combinatorial and recurrence properties of intersective polynomials over $\mZ$.
Of particular note, we show that the set of values generated by an intersective polynomial is a van der Corput set, providing a full resolution to \cite[Conjecture 6.2]{llw}.
(The corresponding result in the integer setting was established by Kamae and Mend\`{e}s France \cite{kmf}. An extension to subsets of $\Z^d$ was obtained in \cite{bl_Z^d}.)

\begin{theorem}
    Let $q(n) \in \mZ[n]$ be a nonconstant polynomial.
    The following are equivalent:
    \begin{enumerate}[(i)]
        \item $q$ is intersective (i.e., $q$ satisfies any/all of (P2), (P3), and (P4));
        \item for any unitary action $U(n)$, $n \in \mZ$, of $\mZ$ on a Hilbert space $\mH$, any $f \in \mH$, and any $\eps > 0$, there exist $m,k \in \mZ$, $m \ne 0$, such that for any F{\o}lner sequence $(\Phi_N)_{N \in \N}$ in $\mZ$,
            \begin{equation*}
		          \lim_{N \to \infty} \frac{1}{|\Phi_N|} \sum_{n \in \Phi_N} \left\langle U(q(mn+k))f, f \right\rangle > \|Pf\|^2 - \eps,
	        \end{equation*}
            where $P$ is the orthogonal projection onto the space of $U$-invariant elements;
        \item $q(\mZ) \setminus \{0\}$ is a set of operatorial recurrence;
        \item $q(\mZ) \setminus \{0\}$ is a van der Corput set;
        \item for any action of $\mZ$ by measure-preserving trasnformations $T(n)$, $n \in \mZ$, on a probability space $(X, \mu)$, any measurable set $A \subseteq X$, and any $\eps > 0$, there exist $m, k \in \mZ$, $m \ne 0$, such that for any F{\o}lner sequence $(\Phi_N)_{N \in \N}$ in $\mZ$,
        \begin{equation*}
		          \lim_{N \to \infty} \frac{1}{|\Phi_N|} \sum_{n \in \Phi_N} \mu \left( A \cap T(-q(mn+k)) A \right) > \mu(A)^2 - \eps;
	        \end{equation*}
        \item for any action of $\mZ$ by measure-preserving trasnformations $T(n)$, $n \in \mZ$, on a probability space $(X, \mu)$, any measurable set $A \subseteq X$, and any $\eps > 0$, the set
        \begin{equation*}
            \left\{ n \in \mZ : \mu(A \cap T(-q(n))A) > \mu(A)^2 - \eps \right\}
        \end{equation*}
        is syndetic;
        \item for any $E \subseteq \mZ$ with $d^*(E) > 0$ and any $\eps > 0$, there exists $m, k \in \mZ$, $m \ne 0$, such that for any F{\o}lner sequence $(\Phi_N)_{N \in \N}$ in $\mZ$,
        \begin{equation*}
            \liminf_{N \to \infty} \frac{1}{|\Phi_N|} \sum_{n \in \Phi_N} d^* \left( E \cap (E - q(mn+k)) \right) > d^*(E)^2 - \eps;
        \end{equation*}
        \item for any $E \subseteq \mZ$ with $d^*(E) > 0$ and any $\eps > 0$, the set
        \begin{equation*}
            \left\{ n \in \mZ : d^* \left( E \cap (E - q(n)) \right) > d^*(E)^2 - \eps \right\}
        \end{equation*}
        is syndetic;
        \item for any finite partition $\mZ = \bigcup_{i=1}^r C_i$, there exists $i \in \{1, \ldots, r\}$ and $x,y,z \in C_i$ such that $x - y = q(z)$.
    \end{enumerate}
\end{theorem}

\begin{proof}
    We carry out the proof by establishing the equivalences displayed in Figure \ref{fig: implications}.
    
    \begin{figure}[h]
        \begin{mdframed}[align = center, userdefinedwidth=25em]
        \begin{center}
	       \begin{tikzcd}[scale cd=0.8]
            (P2) \arrow[r, Leftrightarrow, "\ref{thm: equivalences}"] & (P3) \arrow[r, Leftrightarrow, "\ref{thm: equivalences}"] \arrow[d, Rightarrow] & (P4) \\
            (ix) \arrow[u, Rightarrow] & (ii) \arrow[r, Rightarrow, "D" below] \arrow[dd, Rightarrow] & (iii) \arrow[r, Leftrightarrow, "G" below] & (iv) \arrow[ul, Rightarrow, "G" above right] \\
             & & (vi) \arrow[rd, Rightarrow, "C"] \\
             & (v) \arrow[uul, Rightarrow] \arrow[ur, Rightarrow, "S"] \arrow[dr, Rightarrow, "C" below left] & & (viii) \arrow[uuul, bend right=70, Rightarrow, "D" above right] \\
             & & (vii) \arrow[ur, Rightarrow, "S" below right]
        \end{tikzcd}
        \end{center}
        \end{mdframed}
        \caption{Implications demonstrated in the proof.}
        \label{fig: implications}
    \end{figure}
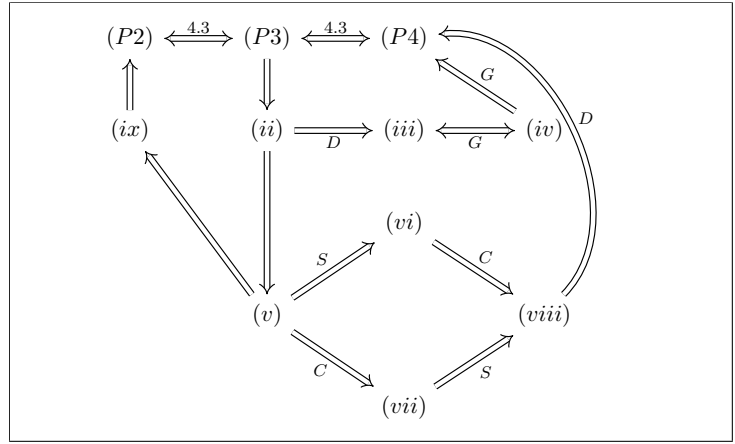
    
    The equivalences in the first row (between (P2), (P3), and (P4)) were established above in Theorem \ref{thm: equivalences}.
    Several implication arrows are labeled by letters to indicate the reason for the implication:
    \begin{itemize}
        \item ``C'' indicates that the implication holds by applying the Furstenberg correspondence principle (\cite[Theorem 4.17]{etdp}),
        \item ``D'' is used to label implications that follow immediately from the relevant definitions (sets of operatorial recurrence in the case of (ii)$\implies$(iii) and intersective sets in the case of (viii)$\implies$(P4)),
        \item ``G'' stands for a general implication between notions as illustrated in Figure \ref{figure}, and
        \item ``S'' means that one property can be deduced from the other using the fact that a set is syndetic if and only if it has nonempty intersection with some element of every F{\o}lner sequence (see \cite[Lemma 1.9]{abb}).
    \end{itemize}

    Let us now show the remaining chain of implications (P3)$\implies$(ii)$\implies$(v)$\implies$(ix)$\implies$(P2).

    (P3)$\implies$(ii).
    The argument used to prove the implication (P3)$\implies$(P4) in the proof of Theorem \ref{thm: equivalences} above applies equally well to arbitrary unitary actions to deduce (ii).

    (ii)$\implies$(v).
    Let $T(n)$, $n \in \mZ$, be an action of $\mZ$ by measure-preserving transformations on a probability space $(X, \mu)$ and $A \subseteq X$ a measurable set.
    Take $\mH = L^2(\mu)$, and let $U(n)f = f \circ T(n)$ for $n \in \mZ$.
    Then $U(n)$, $n \in \mZ$, is a unitary action of $\mZ$ on $\mH$.
    Therefore, given $\eps > 0$, there exist $m, k \in \mZ$, $m \ne 0$, such that for every F{\o}lner sequence $(\Phi_N)_{N \in \N}$ in $\mZ$,
    \begin{equation*}
        \lim_{N \to \infty} \frac{1}{|\Phi_N|} \sum_{n \in \Phi_N} \langle U(q(mn+k))1_A, 1_A \rangle > \|P1_A\|^2 - \eps.
    \end{equation*}
    Now,
    \begin{equation*}
        \langle U(q(mn+k))1_A, 1_A \rangle = \mu \left( A \cap T(-q(mn+k))A \right)
    \end{equation*}
    and
    \begin{equation*}
        \|P1_A\|^2 \ge \mu(A)^2
    \end{equation*}
    so
    \begin{equation*}
        \lim_{N \to \infty} \frac{1}{|\Phi_N|} \sum_{n \in \Phi_N} \mu(A \cap T(-q(mn+k))A) \ge \mu(A)^2 - \eps.
    \end{equation*}

    (v)$\implies$(ix).
    We use the Furstenberg correspondence principle together with the ``coloring trick'' from \cite{density_schur}.
    This argument also appears in \cite[Section 8]{ab_rings}.
    Let $\mZ = \bigcup_{i=1}^r C_i$.
    Fix a F{\o}lner sequence $(\Phi_N)_{N \in \N}$ in $\mZ$.
    Let $L$ be the collection of ``large'' color classes
    \begin{equation*}
        L = \{i \in \{1, \ldots, r\} : \overline{d}_{(\Phi_N)}(C_i) > 0\}
    \end{equation*}
    and $S$ the ``small'' color classes
    \begin{equation*}
        S = \{i \in \{1, \ldots, r\} : \overline{d}_{(\Phi_N)}(C_i) = 0\}.
    \end{equation*}
    Let $E = \prod_{i \in L} C_i \subseteq \mZ^L$.
    For each $i \in L$, we may find a subsequence $(\Phi_{N_{i,k}})_{k \in \N}$ such that the limit
    \begin{equation*}
        d_{(\Phi_{N_{i,k}})}(C_i) = \lim_{k \to \infty} \frac{|C_i \cap \Phi_{N_{i,k}}|}{|\Phi_{N_{i,k}}|}
    \end{equation*}
    exists and is positive.
    Let $\Psi_k = \prod_{i \in L} \Phi_{N_{i,k}}$.
    Then $(\Psi_k)_{k \in \N}$ is a F{\o}lner sequence in $\mZ^L$ with $d_{(\Psi_k)}(E) > 0$.
    Applying the Furstenberg correspondence principle, there exists an action of $\mZ$ by measure-preserving transformations $T(n)$, $n \in \mZ$, on a probability space $(X, \mu)$ and a measurable set $A \subseteq X$ with $\mu(A) = d_{(\Psi_k)}(E) > 0$ such that
    \begin{equation*}
        \overline{d}_{(\Psi_k)} \left( E \cap (E - (n,\ldots,n)) \right) \ge \mu(A \cap T(-n)A)
    \end{equation*}
    for every $n \in \mZ$.
    By (v), the set
    \begin{equation*}
        R = \left\{ n \in \mZ : \mu(A \cap T(-q(n))A) > \frac{\mu(A)^2}{2} \right\}
    \end{equation*}
    is syndetic.
    In particular, $\underline{d}_{(\Phi_N)}(R) > 0$, so $R \cap C_i \ne \es$ for some $i \in L$.
    Let $z \in R \cap C_i$ with $i \in L$.
    Then
    \begin{equation*}
        \overline{d}_{(\Phi_N)}(C_i \cap (C_i - q(z))) \ge \overline{d}_{(\Psi_k)}(E \cap (E - (q(z), \ldots, q(z)))) \ge \mu(A \cap T(-q(n))A) > 0,
    \end{equation*}
    so there exist $x, y \in C_i$ such that $x - y = q(z)$.

    (ix)$\implies$(P2).
    Suppose (ix) holds.
    Given a subgroup $\Lambda \le \mZ$ of finite index, we may consider the partition $\mZ = \bigcup_{i=1}^r (\Lambda + n_i)$ into cosets of $\Lambda$ and deduce that there exists $z \in \mZ$ such that $q(z) \in \Lambda$, so $q$ satisfies (P2).
    This completes the proof.
\end{proof}


\end{document}